\documentclass{amsart}
\usepackage{latexsym,amssymb,amsmath,amsthm,amscd,graphicx}
\usepackage{esint} 
\usepackage{color}
\usepackage{citeref}
\numberwithin{equation}{section}
\newtheorem{theorem}{Theorem}[section]
\newtheorem{lemma}[theorem]{Lemma}

\newtheorem{proposition}[theorem]{Proposition}

\theoremstyle{definition}

\newtheorem{definition}[theorem]{Definition}

\theoremstyle{remark}

\begin{document}

\title[Bilinear estimates on Morrey spaces]{Bilinear estimates on Morrey spaces by using average}
\author{Naoya Hatano}
\address[Naoya Hatano]{Department of Mathematics, Chuo University, 1-13-27, Kasuga, Bunkyo-ku, Tokyo 112-8551, Japan}
\email[Naoya Hatano]{n18012@gug.math.chuo-u.ac.jp}
\maketitle

\begin{abstract}
This paper is a follow up of \cite{HaSa19}. 
We investigate the boundedness of the bilinear fractional integral operator introduced by Grafakos in \cite{Grafakos92}.
When the local integrability index $s$ falls $1$ with weights and $t$ exceeds $1$, 
He and Yan obtained some results
on this operator was worked on Morrey spaces earlier in \cite{HeYa18}. 
Later in the paper \cite{HaSa19}, we considered the case $t=1$. 
This paper handles the remaining case $0<t<1$.
\end{abstract}



\section{Introduction}

The bilinear fractional integral operator of Grafakos type $\mathcal{J}_\alpha$, which is given by
\[
\mathcal{J}_\alpha[f_1,f_2](x)\equiv\int_{\mathbb{R}^n}\frac{f_1(x+y)f_2(x-y)}{|y|^{n-\alpha}}{\rm d}y, \quad x\in\mathbb{R}^n,
\]
for non-negative measurable functions or more general complex-valued measurable functions $f_1,f_2$ subject to a certain integrability condition,
it is known that operator $\mathcal{J}_\alpha$ is bounded bilinear fractional integral operator from product Lebesgue space $L^{p_1}(\mathbb{R}^n)\times L^{p_2}(\mathbb{R}^n)$ to Lebesgue space $L^s(\mathbb{R}^n)$ in \cite{KeSt99}.
We would like to expand this result and show that operator $\mathcal{J}_\alpha$ is bounded bilinear fractional integral operator from product Morrey space $\mathcal{M}^{p_1}_{q_1}(\mathbb{R}^n)\times\mathcal{M}^{p_2}_{q_2}(\mathbb{R}^n)$ to Morrey space $\mathcal{M}^s_t$.
We discuss it under $s<\min(q_1,q_2)$.
In the paper \cite{HaSa19}, when $1\le t\le s<\min(q_1,q_2)$, we obtained the result.
When $0<t\le s<1$, in the paper \cite{HeYa18}, He and Yan show bilinear estimates with weights.
In this paper, we give an alternative proof of non-weighted type of their results.
Furthermore developing this proof,
we obtain the case $0<t\le 1\le s<\min(q_1,q_2)$.

Let $0<q\le p<\infty$. 
Define the {\it Morrey space} ${\mathcal M}^p_q({\mathbb R}^n)$ by
\[
\mathcal{M}^p_q(\mathbb{R}^n)\equiv
\left\{
f\in L^0(\mathbb{R}^n):
\|f\|_{\mathcal{M}^p_q}\equiv\sup_{Q\in\mathcal{D}}|Q|^{\frac1p-\frac1q}\left(\int_Q|f(x)|^qdx\right)^{\frac1q}<\infty
\right\},
\]
where $\mathcal{D}$ denotes the family of all dyadic cubes in $\mathbb{R}^n$. We recall the definition of the dyadic cubes precisely in Section \ref{s2}.

Grafakos introduced  the bilinear fractional integral operator in \cite{Grafakos92}.

\begin{definition}
Let $0<\alpha<n$. Define the {\it bilinear fractional integral operator} by
\[
\mathcal{J}_\alpha[f_1,f_2](x)\equiv\int_{\mathbb{R}^n}\frac{f_1(x+y)f_2(x-y)}{|y|^{n-\alpha}}{\rm d}y, \quad x\in\mathbb{R}^n,
\]
for measurable functions $f_1,f_2$ defined in $\mathbb{R}^n$.
\end{definition}

First,
He and Yan, under the condition $\frac1{q_1}+\frac1{q_2}<1$ and $1<t\le s<\infty$, investigated the operator $\mathcal{J}_\alpha$ acting on Morrey spaces in \cite{HeYa18} earlier.
Next,
we prove the case $1\le t\le s<\min(q_1,q_2)$ in \cite{HaSa19} as follows:

\begin{theorem}\label{thm:190204-18}
Let
\[
0<\alpha<n,\quad
1<q_1 \le p_1<\infty,\quad
1<q_2 \le p_2<\infty,\quad
1\le t \le s<\infty.
\]
Define $p$ and $q$ by
\[
\frac{1}{p}=\frac{1}{p_1}+\frac{1}{p_2}, \quad
\frac{1}{q}=\frac{1}{q_1}+\frac{1}{q_2}, 
\]
Assume that 
\[
\frac{1}{s}=\frac{1}{p}-\frac{\alpha}{n}, \quad
\frac{q}{p}=\frac{t}{s}, \quad
s<\min(q_1,q_2).
\]
Then
for all
$f_1 \in {\mathcal M}^{p_1}_{q_1}({\mathbb R}^n)$
and
$f_2 \in {\mathcal M}^{p_2}_{q_2}({\mathbb R}^n)$,
\[
\|{\mathcal J}_\alpha[f_1,f_2]\|_{{\mathcal M}^s_t}
\lesssim
\|f_1\|_{{\mathcal M}^{p_1}_{q_1}}
\|f_2\|_{{\mathcal M}^{p_2}_{q_2}}.
\]
\end{theorem}

We investigate the boundedness property of ${\mathcal J}_\alpha$ when $0<t\le1$ in this paper.
Two cases must be considered
according to the value of $s$.
In the Theorems 4.2 and 4.6 in \cite{HeYa18} which are boundednesses of the bilinear fractional integral operator on Morrey spaces with weights, taking weights to identical equal to 1 implies the following theorem.
So their results cover the Theorem \ref{thm:190322-2} below.

\begin{theorem}{\cite[Theorems 4.2 and 4.6]{HeYa18}}\label{thm:190322-2}
Let $0<\alpha<n$, $1<q_j\le p_j<\infty$, $0<q\le p<\infty$ and $0<t\le s<1$ for $j=1,2$. Assume that
\begin{align*}
&\frac1p=\frac1{p_1}+\frac1{p_2}, \quad \frac1q=\frac1{q_1}+\frac1{q_2},\\
&\frac1s=\frac1p-\frac{\alpha}n, \quad \frac qp=\frac ts.
\end{align*}
Then we have
\[
\|\mathcal{J}_\alpha[f_1,f_2]\|_{\mathcal{M}^s_t}
\lesssim
\|f_1\|_{\mathcal{M}^{p_1}_{q_1}}\|f_2\|_{\mathcal{M}^{p_2}_{q_2}}
\]
for any $f_1\in\mathcal{M}^{p_1}_{q_1}(\mathbb{R}^n)$ and for any $f_2\in\mathcal{M}^{p_2}_{q_2}(\mathbb{R}^n)$.
\end{theorem}

\begin{theorem}\label{thm:190322-4}
Let $0<\alpha<n$, $1<q_j\le p_j<\infty$, $0<q\le p<\infty$ and $0<t \le 1 \le s$ for $j=1,2$ with
\begin{align*}
&\frac1p=\frac1{p_1}+\frac1{p_2}, \quad \frac1q=\frac1{q_1}+\frac1{q_2},\\
&\frac1s=\frac1p-\frac{\alpha}n, \quad \frac qp=\frac ts, \quad s<\min(q_1,q_2).
\end{align*}
Then we have
\[
\|\mathcal{J}_\alpha[f_1,f_2]\|_{\mathcal{M}^s_t}
\lesssim
\|f_1\|_{\mathcal{M}^{p_1}_{q_1}}\|f_2\|_{\mathcal{M}^{p_2}_{q_2}}
\]
for any $f_1\in\mathcal{M}^{p_1}_{q_1}(\mathbb{R}^n)$ and for any $f_2\in\mathcal{M}^{p_2}_{q_2}(\mathbb{R}^n)$.
\end{theorem}

First we collect some necessary facts in Section \ref{s2},
Next we give an alternative proof of Theorem \ref{thm:190322-2} in Section \ref{s3},
Finally Theorem \ref{thm:190322-4} is proved in Section \ref{s4}.

\section{Preliminaries}
\label{s2}

For $f\in L^1_{\rm loc}(\mathbb{R}^n)$, define a function $Mf$ by
\[
Mf(x)\equiv\sup_Q\frac{\chi_Q(x)}{|Q|}\int_Q|f(y)|{\rm d}y, \quad x\in\mathbb{R}^n,
\]
where the supremum is over all cubes in $\mathbb{R}^n$. This is called the {\it Hardy-Littlewood maximal function}. And  the $\eta$-{\it powered Hardy-Littlewood maximal function} is defined by $M^{(\eta)}f\equiv\left(M[|f|^\eta]\right)^{\frac1{\eta}}$, for $f\in L^\eta_{\rm loc}(\mathbb{R}^n)$. 
In \cite{ChFr87}
Chiarenza and Frasca 
showed that
\[
\|M^{(\eta)}f\|_{\mathcal{M}^p_q}
\lesssim
\|f\|_{\mathcal{M}^p_q}, \quad f\in\mathcal{M}^p_q(\mathbb{R}^n)
\]
if $0<\eta<q\le p<\infty$.
For a dyadic cube $Q$ in $\mathbb{R}^n$, 
an average $m_Q(f)$ over $Q$ is defined by
\[
m_Q(f)\equiv\frac1{|Q|}\int_Q f(x){\rm d}x,
\]
for an integrable function $f$. 
Moreover, we write $m_Q^{(\eta)}(f)\equiv\left(m_Q(|f|^\eta)\right)^{\frac1\eta}$ for $\eta>0$
and a measurable function $f$. 
A dyadic cube is a set of the form $Q_{jk}\equiv\displaystyle\prod_{i=1}^n[k_i2^{-j},(k_i+1)2^{-j})$ for some $j\in\mathbb{Z}$, $k=(k_1,\ldots,k_n)\in\mathbb{Z}^n$, and define as follows:
\begin{align*}
&\mathcal{D}_j\equiv\{Q_{jk}:k\in\mathbb{Z}^n\}, \quad j\in\mathbb{Z},\\
&\mathcal{D}\equiv\{Q_{jk}:j\in\mathbb{Z}, k\in\mathbb{Z}^n\}=\bigcup_{j\in\mathbb{Z}}\mathcal{D}_j.
\end{align*}
We show a key estimate which is interesting of its own right 
by using the following proposition:

\begin{proposition}{\rm\cite[Lemma 3.1]{GrKa01}}\label{prop:190322-1}
Let $0<p \le 1$.
Then
\[
\left\|\sum_{j=1}^\infty f_j\right\|_{L^p}
\lesssim
\left\|\sum_{j=1}^\infty m_{Q_j}(f_j)\chi_{Q_j}\right\|_{L^p}
\] for all non-negative sequences
$\{f_j\}_{j=1}^\infty$
of integrable functions
such that 
each $f_j$ is supported on a cube $Q_j$.
\end{proposition}
This proposition implies the following theorem.

\begin{theorem}\label{thm:190322-1}
Let $0<q \le p<1$.
Then
\[
\left\|\sum_{j=1}^\infty f_j\right\|_{{\mathcal M}^p_q}
\lesssim
\left\|\sum_{j=1}^\infty m_{Q_j}(f_j)\chi_{Q_j}\right\|_{{\mathcal M}^p_q}
\] 
for all non-negative sequences
$\{f_j\}_{j=1}^\infty$
of integrable functions
such that 
each $f_j$ is supported on a dyadic cube $Q_j$.
\end{theorem}

\begin{proof}
Fix a dyadic cube $Q$.
It suffices to show that
\begin{align*}
|Q|^{\frac1p-\frac1q}
\left\|\sum_{j=1}^\infty \chi_{Q}f_j\right\|_{L^q}
\lesssim
\left\|\sum_{j=1}^\infty m_{Q_j}(f_j)\chi_{Q_j}\right\|_{{\mathcal M}^p_q}.
\end{align*}
We split this estimate into two parts:
\begin{align*}
|Q|^{\frac1p-\frac1q}
\left\|\sum_{j \in {\mathbb N}, Q_j \subset Q} \chi_{Q}f_j\right\|_{L^q}
\lesssim
\left\|\sum_{j=1}^\infty m_{Q_j}(f_j)\chi_{Q_j}\right\|_{{\mathcal M}^p_q}
\end{align*}
and
\begin{align*}
|Q|^{\frac1p-\frac1q}
\left\|\sum_{j \in {\mathbb N}, Q_j \supset Q} \chi_{Q}f_j\right\|_{L^q}
\lesssim
\left\|\sum_{j=1}^\infty m_{Q_j}(f_j)\chi_{Q_j}\right\|_{{\mathcal M}^p_q}.
\end{align*}
The first estimate is a consequence
of Proposition \ref{prop:190322-1}.
Let us prove the second inequality.

We will show
\begin{align*}
|Q|^{\frac{q}p-1}
\left(\left\|\sum_{j \in {\mathbb N}, Q_j \supset Q} \chi_{Q}f_j\right\|_{L^q}
\right)^q
\lesssim
\left(
\left\|\sum_{j=1}^\infty m_{Q_j}(f_j)\chi_{Q_j}\right\|_{{\mathcal M}^p_q}
\right)^q.
\end{align*}
By the $q$-triangle inequality
\begin{align*}
|Q|^{\frac{q}p-1}
\left(\left\|\sum_{j \in {\mathbb N}, Q_j \supset Q} \chi_{Q}f_j\right\|_{L^q}
\right)^q
&\le
|Q|^{\frac{q}p-1}
\sum_{j \in {\mathbb N}, Q_j \supset Q}
\left(\left\| \chi_{Q}f_j\right\|_{L^q}
\right)^q\\
&\le
|Q|^{\frac{q}p-1}
\sum_{j \in {\mathbb N}, Q_j \supset Q}
\left(|Q|^{\frac1q-1}\left\| \chi_{Q}f_j\right\|_{L^1}
\right)^q.
\end{align*}
We fix $J \in {\mathbb N}$.
Then we have
\[
\left\|m_{Q_J}(f_J)\chi_{Q_J}\right\|_{{\mathcal M}^p_q}
\le
\left\|\sum_{j=1}^\infty m_{Q_j}(f_j)\chi_{Q_j}\right\|_{{\mathcal M}^p_q},
\]
or equivalently,
\[
\|f_J\|_{L^1} \le |Q_J|^{1-\frac1p}
\left\|\sum_{j=1}^\infty m_{Q_j}(f_j)\chi_{Q_j}\right\|_{{\mathcal M}^p_q}.
\]
Consequently,
\begin{align*}
|Q|^{\frac{q}p-1}
\left(\left\|\sum_{j \in {\mathbb N}, Q_j \supset Q} \chi_{Q}f_j\right\|_{L^q}
\right)^q
&\le
\left\|\sum_{j=1}^\infty m_{Q_j}(f_j)\chi_{Q_j}\right\|_{{\mathcal M}^p_q}
|Q|^{\frac{q}p-1}
\sum_{j \in {\mathbb N}, Q_j \supset Q}
|Q_j|^{q-\frac{q}p}|Q|^{1-q}\\
&\lesssim
\left\|\sum_{j=1}^\infty m_{Q_j}(f_j)\chi_{Q_j}\right\|_{{\mathcal M}^p_q}.
\end{align*}
This completes the proof.
\end{proof}
Similarly,
the estimate of $u$-powered type holds, too.

\begin{theorem}{\rm \cite{IST14}}\label{thm:190322-3}
Let $0<q \le 1 \le p<u$.
Then
\[
\left\|\sum_{j=1}^\infty f_j\right\|_{{\mathcal M}^p_q}
\lesssim
\left\|\sum_{j=1}^\infty m_{Q_j}^{(u)}(f_j)\chi_{Q_j}\right\|_{{\mathcal M}^p_q}
\] 
for all non-negative sequences
$\{f_j\}_{j=1}^\infty$
of integrable functions
such that 
each $f_j$ is supported on a dyadic cube $Q_j$.
\end{theorem}

We obtain the following lemma similar to  in \cite[Lemma 2.1]{HaSa19}.

\begin{lemma}\label{lem:180615-41}
Let $f_1,f_2 \ge 0$ be measurable functions.
Then we have
\begin{align*}
{\mathcal J}_{\alpha}[f_1,f_2](x)
&\lesssim
\sum_{j=-\infty}^\infty\sum_{Q \in {\mathcal D}_j}
\frac{\chi_Q(x)}{\ell(Q)^{n-\alpha}}
\int_{B(2^{-j})}
f_1(x+y)f_2(x-y){\rm d}y
\quad (x \in {\mathbb R}^n).
\end{align*}
\end{lemma}

In addition, we use \cite[Lemma 2.3]{HaSa19}. For the bilinear fractional integral operator $\mathcal{I}_\alpha$, $0<\alpha<2n$ defined by
\[
\mathcal{I}_\alpha[f_1,f_2](x)\equiv\int_{\mathbb{R}^n}\frac{f_1(y_1)f_2(y_2)}{(|x-y_1|+|x-y_2|)^{2n-\alpha}}{\rm d}y_1{\rm d}y_2, \quad x\in\mathbb{R}^n,
\]
for integrable functions $f_1,f_2$,
introduced by Kenig and Stein \cite{KeSt99}, we obtain the pointwise estimate
\[
\mathcal{I}_\alpha[f_1,f_2](x)
\lesssim
\sum_{Q \in {\mathcal D}}
\frac{\chi_Q(x)}{\ell(Q)^{2n-\alpha}}
\int_{(3Q)^2}f_1(y_1)f_2(y_2){\rm d}y_1{\rm d}y_2,
\]
for non-negative measurable functions $f_1,f_2$ in \cite[Lemma 4.1]{ISST12}. Therefore using following lemma, we show that $\mathcal{I}_\alpha$ is bounded from $\mathcal{M}^{p_1}_{q_1}\times\mathcal{M}^{p_2}_{q_2}$ to $\mathcal{M}^s_t$.

\begin{lemma}\label{lem:190322-1}
Let
\[
0<\alpha<2 n, \quad
1<q_j \le p_j<\infty, \quad
0<q \le p<\infty, \quad
0<t \le s<\infty
\]
for $j=1,2$.
Assume
\[
\frac{1}{p}=\frac{1}{p_1}+\frac{1}{p_2}, \quad
\frac{1}{q}=\frac{1}{q_1}+\frac{1}{q_2}, 
\]
\[
\frac{1}{s}=\frac{1}{p}-\frac{\alpha}{n}, \quad
\frac{q}{p}=\frac{t}{s}.
\]
Then
\begin{equation*}
\left\|
\sum_{Q \in {\mathcal D}}
\frac{\chi_Q}{\ell(Q)^{2n-\alpha}}
\int_{(3Q)^2}f_1(y_1)f_2(y_2){\rm d}y_1{\rm d}y_2
\right\|_{\mathcal{M}^s_t}
\lesssim
\prod_{j=1}^2 \|f_j\|_{{\mathcal M}^{p_j}_{q_j}}
\end{equation*}
for all non-negative measurable functions 
$f_1,f_2$.
\end{lemma}

By similar to prove Lemma \ref{lem:190322-1}, we obtain the result of $u$-powered type.

\begin{lemma}\label{lem:190328-1}
Let
\[
0<\alpha<2 n, \quad
1<q_j \le p_j<\infty, \quad
0<q \le p<\infty, \quad
0<t \le s<\infty, \quad
0<u<\infty
\]
for $j=1,2$.
Assume
\[
\frac{1}{p}=\frac{1}{p_1}+\frac{1}{p_2}, \quad
\frac{1}{q}=\frac{1}{q_1}+\frac{1}{q_2}, 
\]
\[
\frac{1}{s}=\frac{1}{p}-\frac{\alpha}{n}, \quad
\frac{q}{p}=\frac{t}{s}, \quad
s<u<\min(q_1,q_2).
\]
Then we have
\begin{equation*}
\left\|
\sum_{Q \in {\mathcal D}}\frac{\chi_Q}{\ell(Q)^{\frac{2n}u-\alpha}}\left(\int_{(3Q)^2}(f_1(y_1)f_2(y_2))^u{\rm d}y_1{\rm d}y_2\right)^{\frac1u}
\right\|_{\mathcal{M}^s_t}
\lesssim
\prod_{j=1}^2 \|f_j\|_{{\mathcal M}^{p_j}_{q_j}}
\end{equation*}
for all non-negative measurable functions 
$f_1,f_2$.
\end{lemma}

\begin{proof}
Let $L=L(x)$ be a  positive number that is specified shortly. Similar to a method by Hedberg in \cite{Hedberg72}, We decompose
\begin{align*}
&\sum_{Q \in {\mathcal D}}\frac{\chi_Q(x)}{\ell(Q)^{\frac{2n}u-\alpha}}\left(\int_{(3Q)^2}|f_1(y_1)f_2(y_2)|^u{\rm d}y_1{\rm d}y_2\right)^{\frac1u}\\
&\hspace{1cm}=
\sum_{Q \in {\mathcal D},\ell(Q)\le L}\frac{\chi_Q(x)}{\ell(Q)^{\frac{2n}u-\alpha}}\left(\int_{(3Q)^2}|f_1(y_1)f_2(y_2)|^u{\rm d}y_1{\rm d}y_2\right)^{\frac1u}\\
&\hspace{1cm}\quad+
\sum_{Q \in {\mathcal D},\ell(Q)>L}\frac{\chi_Q(x)}{\ell(Q)^{\frac{2n}u-\alpha}}\left(\int_{(3Q)^2}|f_1(y_1)f_2(y_2)|^u{\rm d}y_1{\rm d}y_2\right)^{\frac1u}\\
&\hspace{1cm}=:
S_1+S_2.
\end{align*}
First, we estimate the quantity $S_1$.
\[
S_1
\lesssim
L^\alpha M^{(u)}f_1(x)M^{(u)}f_2(x).
\]
Next, we estimate the quantity $S_2$. By H\"{o}lder's inequality,
\begin{align*}
S_2
&\lesssim
\sum_{Q \in {\mathcal D},\ell(Q)>L}\frac{\chi_Q(x)}{\ell(Q)^{\frac{2n}u-\alpha}}\left(|Q|^{\frac1{(q_1/u)'}}\||f_1|^u\|_{L^{\frac{q_1}u}(3Q)}\right)^{\frac1u}
\left(|Q|^{\frac1{(q_2/u)'}}\||f_2|^u\|_{L^{\frac{q_2}u}(3Q)}\right)^{\frac1u}\\
&=
\sum_{Q \in {\mathcal D},\ell(Q)>L}\chi_Q(x)|Q|^{\frac\alpha n-\frac1{q_1}-\frac1{q_2}}\|f_1\|_{L^{q_1}(3Q)}\|f_2\|_{L^{q_2}(3Q)}\\
&\lesssim
L^{-\frac ns}\|f_1\|_{\mathcal{M}^{p_1}_{q_1}}\|f_2\|_{\mathcal{M}^{p_2}_{q_2}}
\end{align*}
Hence we obtain
\begin{align*}
&\sum_{Q \in {\mathcal D}}\frac{\chi_Q(x)}{\ell(Q)^{\frac{2n}u-\alpha}}\left(\int_{(3Q)^2}|f_1(y_1)f_2(y_2)|^u{\rm d}y_1{\rm d}y_2\right)^{\frac1u}\\
&\hspace{2cm}\lesssim
L^\alpha M^{(u)}f_1(x)M^{(u)}f_2(x)
+L^{-\frac ns}\|f_1\|_{{\mathcal M}^{p_1}_{q_1}}\|f_2\|_{{\mathcal M}^{p_2}_{q_2}}.
\end{align*}
In particular, choose the constant $L=L(x)$ 
to optimize the right-hand side:
\[
L=\left(\frac{\|f_1\|_{{\mathcal M}^{p_1}_{q_1}}\|f_2\|_{{\mathcal M}^{p_2}_{q_2}}}{M^{(u)}f_1(x)M^{(u)}f_2(x)}\right)^{\frac pn}.
\]
Then we have
\begin{align*}
&\sum_{Q \in {\mathcal D}}\frac{\chi_Q(x)}{\ell(Q)^{\frac{2n}{u}-\alpha}}\left(\int_{(3Q)^2}|f_1(y_1)f_2(y_2)|^u{\rm d}y_1{\rm d}y_2\right)^{\frac1u}\\
&\hspace{2cm}\lesssim
(M^{(u)}f_1(x)M^{(u)}f_2(x))^{\frac ps}\left(\|f_1\|_{{\mathcal M}^{p_1}_{q_1}}\|f_2\|_{{\mathcal M}^{p_2}_{q_2}}\right)^{1-\frac ps}.
\end{align*}
Therefore, using H$\ddot{\rm o}$lder's inequality
for Morrey spaces,
the ${\mathcal M}^{p_1}_{q_1}({\mathbb R}^n)$-boundedness
of $M^{(u)}$
and
the ${\mathcal M}^{p_1}_{q_1}({\mathbb R}^n)$-boundedness
of $M^{(u)}$,
we have
\begin{align*}
&
\left\|
\sum_{Q \in {\mathcal D}}\frac{\chi_Q}{\ell(Q)^{\frac{2n}u-\alpha}}\left(\int_{(3Q)^2}|f_1(y_1)f_2(y_2)|^u{\rm d}y_1{\rm d}y_2\right)^{\frac1u}
\right\|_{\mathcal{M}^s_t}\\
&\hspace{2cm}\lesssim\left\|(M^{(u)}f_1\cdot M^{(u)}f_2)^{\frac ps}\right\|_{{\mathcal M}^s_t}\left(\|f_1\|_{{\mathcal M}^{p_1}_{q_1}}\|f_2\|_{{\mathcal M}^{p_2}_{q_2}}\right)^{1-\frac ps}\\
&\hspace{2cm}=\|M^{(u)}f_1\cdot M^{(u)}f_2\|_{{\mathcal M}^p_q}^{\frac ps}\left(\|f_1\|_{{\mathcal M}^{p_1}_{q_1}}\|f_2\|_{{\mathcal M}^{p_2}_{q_2}}\right)^{1-\frac ps}\\
&\hspace{2cm}\le\left(\|M^{(u)}f_1\|_{{\mathcal M}^{p_1}_{q_1}}\|M^{(u)}f_2\|_{{\mathcal M}^{p_2}_{q_2}}\right)^{\frac ps}\left(\|f_1\|_{{\mathcal M}^{p_1}_{q_1}}\|f_2\|_{{\mathcal M}^{p_2}_{q_2}}\right)^{1-\frac ps}\\
&\hspace{2cm}\lesssim\|f_1\|_{{\mathcal M}^{p_1}_{q_1}}\|f_2\|_{{\mathcal M}^{p_2}_{q_2}}.
\end{align*}
\end{proof}

\section{Proof of Theorem \ref{thm:190322-2}}
\label{s3}

For each $Q\in\mathcal{D}_j$, we abbreviate
\[
F_{j,Q}\equiv\frac{\chi_Q}{\ell(Q)^{n-\alpha}}\int_{B(2^{-j})}f_1(\cdot+y)f_2(\cdot-y){\rm d}y.
\]
We calculate its average  $m_Q(F_{j,Q})$ as follows:
\begin{align*}
m_Q(F_{j,Q})
&=
\frac1{\ell(Q)^{n-\alpha}}\fint_Q\int_{B(2^{-j})}f_1(x+y)f_2(x-y){\rm d}y{\rm d}x\\
&\lesssim
\frac1{\ell(Q)^{2n-\alpha}}\int_{(3Q)^2}f_1(y_1)f_2(y_2){\rm d}y_1{\rm d}y_2.
\end{align*}
Hence by Theorem \ref{thm:190322-1} and Lemma \ref{lem:180615-41},
\begin{align*}
\|\mathcal{J}_\alpha[f_1,f_2]\|_{\mathcal{M}^s_t}
&\lesssim
\left\|\sum_{j=-\infty}^\infty\sum_{Q\in\mathcal{D}_j}F_{j,Q}\right\|_{\mathcal{M}^s_t}\\
&\lesssim
\left\|\sum_{j=-\infty}^\infty\sum_{Q\in\mathcal{D}_j}m_Q(F_{j,Q})\chi_Q\right\|_{\mathcal{M}^s_t}\\
&\lesssim
\left\|\sum_{j=-\infty}^\infty\sum_{Q\in\mathcal{D}_j}\frac{\chi_Q}{\ell(Q)^{2n-\alpha}}\int_{(3Q)^2}f_1(y_1)f_2(y_2){\rm d}y_1{\rm d}y_2\right\|_{\mathcal{M}^s_t}.
\end{align*}
Finally using Lemma \ref{lem:190322-1}, we obtain the result.

\section{Proof of Theorem \ref{thm:190322-4}}
\label{s4}

Define $F_{j,Q}$ as above, and fix a parameter $u$ such that $s<u<\min(q_1,q_2)$.
Thus we calculate its average $m^{(u)}_Q(F_{j,Q})$ as follows:
\begin{align*}
m^{(u)}_Q(F_{j,Q})
&=
\frac1{\ell(Q)^{n-\alpha}}\left\{\fint_Q\left(\int_{B(2^{-j})}f_1(x+y)f_2(x-y){\rm d}y\right)^u{\rm d}x\right\}^{\frac1u}\\
&\le
\frac1{\ell(Q)^{n-\alpha}}\int_{B(2^{-1})}\left(\fint_Q|f_1(x+y)f_2(x-y)|^u{\rm d}x\right)^{\frac1u}{\rm d}y\\
&\le
\frac1{\ell(Q)^{n-\alpha}}|B(2^{-j})|^{\frac1{u'}}\left(\int_{B(2^{-j})}\fint_Q|f_1(x+y)f_2(x-y)|^u{\rm d}x{\rm d}y\right)^{\frac1u}\\
&\lesssim
\frac1{\ell(Q)^{\frac{2n}u-\alpha}}\left(\int_{(3Q)^2}|f_1(y_1)f_2(y_2)|^u{\rm d}y_1{\rm d}y_2\right)^{\frac1u}.
\end{align*}
Hence by Theorem \ref{thm:190322-3} and Lemma \ref{lem:180615-41},
\begin{align*}
\|\mathcal{J}_\alpha[f_1,f_2]\|_{\mathcal{M}^s_t}
&\lesssim
\left\|\sum_{j=-\infty}^\infty\sum_{Q\in\mathcal{D}_j}F_{j,Q}\right\|_{\mathcal{M}^s_t}\\
&\lesssim
\left\|\sum_{j=-\infty}^\infty\sum_{Q\in\mathcal{D}_j}m^{(u)}_Q(F_{j,Q})\chi_Q\right\|_{\mathcal{M}^s_t}\\
&\lesssim
\left\|\sum_{j=-\infty}^\infty\sum_{Q\in\mathcal{D}_j}\frac{\chi_Q}{\ell(Q)^{\frac{2n}u-\alpha}}\left(\int_{(3Q)^2}|f_1(y_1)f_2(y_2)|^u{\rm d}y_1{\rm d}y_2\right)^{\frac1u}\right\|_{\mathcal{M}^s_t}.
\end{align*}
Finally using Lemma \ref{lem:190328-1}, we obtain the result.

\section*{Acknowledgement}

The author would like to be thankful
to his advisor Professor Kotaro Tsugawa
for his guidance and advise, 
and be grateful to Professor Yoshihiro Sawano, in Tokyo Metropolitan University, 
for his many kinds of ideas and answering many questions.
In particular, Sawano gave him a hint to Theorem \ref{thm:190322-1}.

\end{document}